\documentclass[12pt]{amsart}
\usepackage{amsmath}
\usepackage{mathtools}
\DeclarePairedDelimiter{\abs}{\lvert}{\rvert}

\DeclarePairedDelimiter{\set}{\{}{\}}

\usepackage{amsfonts}
\usepackage{amsthm}
\usepackage{autobreak}
\usepackage[english]{babel}
\usepackage{bold-extra}
\usepackage{hyperref}
\newtheorem{definition}{Definition}
\newtheorem{prop}[definition]{Proposition}
\newtheorem{cor}[definition]{Corollary}

\newtheorem{lemma}[definition]{Lemma}
\newtheorem{theorem}[definition]{Theorem}
\newtheorem*{maintheorem}{Main Theorems}
\newtheorem*{mainproblem}{Main Problems}

\newtheorem*{questionsss}{Questions}

\newtheorem{remark}[definition]{Remark}

\numberwithin{definition}{section}
\allowdisplaybreaks

\def\D{{\mathbb{D}}}

\def\C{{\mathbb{C}}}
\def\Q{{\mathbb{Q}}}
\def\N{{\mathbb{N}}}

\def\T{{\mathbb{T}}}
\def\UU{{\mathcal{U}}}

\def\VV{{\mathcal{V}}}

\def\e{\varepsilon}

\newcommand{\vp}{\varphi}

\usepackage{multicol}
\usepackage[T1]{fontenc}
\usepackage[utf8]{inputenc}
\usepackage[a4paper]{geometry}
\usepackage{amssymb,amsmath}
\usepackage{version}
\usepackage{mathtools} 
\usepackage{enumerate}
\usepackage{dsfont}
\usepackage{mathdots}
\makeatletter

\usepackage{latexsym}
\usepackage{color}

\begin{document}
	
	\title{Invariance of Abel universality under composition and applications}
	
	\author{St\'{e}phane Charpentier,  Myrto Manolaki, Konstantinos Maronikolakis}
	
	\address{St\'ephane Charpentier, Institut de Math\'ematiques, UMR 7373, Aix-Marseille
		Universite, 39 rue F. Joliot Curie, 13453 Marseille Cedex 13, France}
	\email{stephane.charpentier.1@univ-amu.fr}
	
	\address{Myrto Manolaki, UCD School of Mathematics and Statistics, University College Dublin, Belfield, Dublin 4, Ireland}
	\email{arhimidis8@yahoo.gr}
	
	\address{Konstantinos Maronikolakis, UCD School of Mathematics and Statistics, University College Dublin, Belfield, Dublin 4, Ireland}
	\email{conmaron@gmail.com}
	
	\thanks{The research conducted in this paper was funded by the Irish Research Council under the Ulysses 2022 Scheme. Konstantinos Maronikolakis acknowledges financial support from the Irish Research Council through the grant [GOIPG/2020/1562].}
	

	\begin{abstract}
		A holomorphic function $f$ on the unit disc $\D$ belongs to the class $\mathcal{U}_A (\mathbb{D})$ of Abel universal functions if the family $\{f_r: 0\leq r<1\}$ of its dilates $f_r(z):=f(rz)$ is dense in the Banach space of all continuous functions on $K$, endowed with the supremum norm, for any proper compact subset $K$ of the unit circle. We prove that this property is invariant under composition from the left with any non-constant entire function. As an application, we show that $\mathcal{U}_A (\mathbb{D})$ is strongly-algebrable. Furthermore, we prove that Abel universality is invariant under composition from the right with an automorphism $\Phi$ of $\D$ if and only if $\Phi$ a rotation. On the other hand, we establish the existence of a subset of $\mathcal{U}_A (\mathbb{D})$ which is residual in the space of holomorphic functions on  $\mathbb{D}$ and is invariant under composition from the right with any automorphism of $\D$.
	\end{abstract}

	\maketitle

	\section{Introduction}

	Let $\D=\{z\in \C:\,|z|<1\}$ and $\T=\partial \D$ denote the unit disc and the unit circle respectively, and let $H(\mathbb{D})$ be the space of holomorphic functions on $\mathbb{D}$, endowed with the topology of local uniform convergence.  This paper is concerned with a class of functions in $H(\mathbb{D})$ which have an extremely chaotic radial behaviour.

	\begin{definition}[Abel universal functions]\label{def-Abel-univ-funct-bis} Let $\rho=(r_n)_n$ be an increasing sequence in $[0,1)$ such that $r_n\to 1$. We denote by $\UU_A(\D,\rho)$ the set of all functions $f\in H(\D)$ that satisfy the following property: for any compact subset $K$ of $\T$, different from $\T$, and any continuous function $\varphi$ on $K$, there exists an increasing sequence $(n_k)_k$ in $\N$ such that
		\[
		\sup_{\zeta\in K}|f(r_{n_k}\zeta) - \varphi(\zeta)| \to 0,\quad k\to \infty.
		\]
		A function in $H(\D)$ is called an Abel universal function if it belongs to $\UU_A(\D,\rho)$ for some $\rho$, and the class of all Abel universal functions is denoted by $\UU_A(\D)$.
	\end{definition}
	
	In \cite{Charpentier2020}, it was shown that $\mathcal{U}_A (\mathbb{D},\rho)$ is a dense $G_{\delta}$ subset of $H(\mathbb{D})$ for any $\rho$ as above, which unified several previous results in \cite{KierstSzpilrajn1933}, \cite{BoivinGauthierParamonov2002}, \cite{BernalGonzalezCalderonMorenoPradoBassas2004} and \cite{Bayart2005}. A reformulation of Definition \ref{def-Abel-univ-funct-bis} reads as follows: a function $f$ in $H(\D)$ is Abel universal if, for any compact subset $K$ of $\T$, different from $\T$, the family $\{f_r:\,0\leq r<1\}$ of its \textit{dilates} $f_r(z):=f(rz)$, $z\in \overline{\D}$, is dense in the Banach space $C(K)$ of complex-valued continuous functions on $K$, endowed with the supremum norm. Dilation is one of the most standard techniques used in complex function theory, and especially in the study of the boundary behaviour of functions in $H(\D)$. For a detailed account of results over the past century and the wider importance of dilation theory we refer the reader to the recent survey \cite{MashreghiDilation2022}. The other standard technique for studying the boundary behaviour of a function $f$ in $H(\D)$ involves considering the partial sums $S_n(f,0)$ of its Taylor expansion about the origin. Thus, the class $\mathcal{U}_A (\mathbb{D})$ of Abel universal functions is the natural analogue of the well-studied class $\mathcal{U}_T(\mathbb{D})$ of \textit{universal Taylor series}, which consists of all functions $f\in H(\D)$ such that the set
	$\{S_n(f,0): n\in \mathbb{N}\}$ is dense in the space of functions in $C(K)$ that are holomorphic on the interior of $K$, for any compact set $K$ in $\mathbb{C}\setminus \mathbb{D}$ with connected complement. In \cite{Nestoridis1996}, Nestoridis showed that $\mathcal{U}_T(\mathbb{D})$ is a dense $G_{\delta}$ subset of $H(\mathbb{D})$, which opened an interesting avenue of investigation, connecting the classical theory of \textit{overconvergence} with the more recent theory of \textit{universality}. 
	
	Since the introduction of Abel universal functions in 2020, several variants and properties have been intensively studied, such as their local and global growth, their boundary behaviour and the behaviour of their Taylor polynomials outside $\mathbb{D}$ (see \cite{Charpentier2020, CharpentierKosinski2021, CharpManoMaroI, CharpentierMouze2022, Maronikolakis2022, Meyrath2022}). It turns out that the notion of Abel universal functions is somehow more flexible than that of universal Taylor series. This especially appears when we carefully look at the construction of an Abel universal function, and when we try to build one that satisfies additional properties. For example, such a construction can be achieved in order to prove that the set $\mathcal{U}_A (\mathbb{D},\rho)$ is not invariant under derivation \cite{CharpentierMouze2022}, while the same problem is still open for universal Taylor series. At the same time, Abel universal functions share many boundary properties with universal Taylor series. For instance, any Abel universal function $f$ admits infinity as an asymptotic value and satisfies the local Picard property; that is, for every region of the form $D_{\zeta,r}:=\{z\in \mathbb{D}: |z-\zeta |<r \}$, where $r>0$ and $\zeta\in\T$, the image $f(D_{\zeta,r})$ is the entire complex plane $\mathbb{C}$ except possibly a point \cite{CharpManoMaroI}. (The analogues of these results for universal Taylor series were proved in \cite{Gardiner2014} and \cite{GardinerKhavinson2014}.)
	Two questions naturally arise: (i) Can the exceptional value in the local Picard property occur? (ii) Can Abel universal functions have finite asymptotic values? 
	
	In this paper, we will derive answers to both questions, by looking at the bigger picture; namely, we will examine the invariance of Abel universality under composition.
	
	\begin{mainproblem}  Let $\rho =(r_n)_n$ be a sequence in $[0,1)$ with $r_n\to 1$ as $n\to \infty$.
		\begin{enumerate}
			\item \label{1}Given $f$ in $\UU_A(\D,\rho)$, describe the holomorphic functions $g$ on $f(\D)$ such that $g\circ f$ belongs to $\UU_A(\D,\rho)$;
			\item \label{2}Describe the holomorphic functions $\Phi:\D\to \D$ which are continuous on $\overline{\D}$, such that $f\circ \Phi$ belongs to $\UU_A(\D,\rho)$, for every $f$ in $\UU_A(\D,\rho)$.
		\end{enumerate}
	\end{mainproblem}

	As an application of one of our main results, we will deduce that if $f\in\mathcal{U}_A (\mathbb{D},\rho)$ then $e^{f}$ is always in $\mathcal{U}_A (\mathbb{D},\rho)$, which shows that question (i) has an affirmative answer. Moreover, if $f$ is in $\mathcal{U}_A (\mathbb{D},\rho)$ and has no zeros on $\mathbb{D}$, we will derive that $1/f$ is always in $\mathcal{U}_A (\mathbb{D},\rho)$, which shows that question (ii) also has an affirmative answer. We note that although the analogues of questions (i) and (ii) for the class $\mathcal{U}_T(\mathbb{D})$ have an affirmative answer as well, the situation concerning invariance under composition is far from being understood. The only known results about invariance from the left are due to Costakis and Melas \cite{CostakisMelas2000}, who showed that there exists a function $f\in\mathcal{U}_T(\mathbb{D})$ such that $e^{f}$ is in $\mathcal{U}_T (\mathbb{D})$, and that there exists a zero-free function $f\in\mathcal{U}_T(\mathbb{D})$ such that $1/f\in\mathcal{U}_T(\mathbb{D})$. Thus, Problem \eqref{1} remains open for $\mathcal{U}_T(\mathbb{D})$, even in simple cases. The same happens for Problem \eqref{2}. The only known result in this direction, is due to Gardiner \cite{Gardiner2014}, who used a mixture of advanced tools from potential theory to show that a certain class of universal Taylor series (defined on a strip) is not conformally invariant. For this result, the conformal mapping was from the disc to a strip, and this was crucial for the deep methods that were used. Hence, even the simple case of composition from the right with a conformal map from $\mathbb{D}$ to $\mathbb{D}$, remains unsolved for universal Taylor series. 
	
	In this paper, we fully address Problems \eqref{1} and \eqref{2}. In particular, we prove the following:
	\begin{maintheorem} Let $\rho =(r_n)_n$ be a sequence in $[0,1)$ with $r_n\to 1$ as $n\to \infty$.
		\begin{enumerate}
			\item For any $f\in \UU_A(\D,\rho)$ and any non-constant holomorphic function $g$ on $f(\D)$ the composition $g\circ f$ belongs to $\UU_A(\D,\rho)$;
			\item Let $\Phi:\D\to \D$ be holomorphic on $\D$ and continuous on $\overline{\D}$. Then $f\circ \Phi$ belongs to $\UU_A(\D,\rho)$ for any $f\in \UU_A(\D,\rho)$, if and only if $\Phi$ is a rotation; that is, $\Phi(z)= e^{i\theta}z$ for some $\theta \in [0,2\pi]$.
		\end{enumerate}
	\end{maintheorem}
	
	The proof of Theorem \eqref{1} relies on a variation of a path lifting theorem due to Iversen \cite{Iversen1914}, that we shall prove using an extension of  the Gross Star Theorem due to Kaplan \cite{Kaplan1954}. The proof of Theorem \eqref{2} (when $\Phi$ is not a rotation) is based on an inductive construction argument.
	
	Theorem \eqref{1} has a nice application to the algebraic structure of the set $\UU_A(\D,\rho)$. We note that $\UU_A(\D,\rho)\cup\{0\}$ is not a linear space (indeed for any $f\in\UU_A(\D,\rho)$ and any polynomial $P$, the function $f+P\in\UU_A(\D,\rho)$). Nonetheless, $\UU_A(\D,\rho)\cup\{0\}$ contains a linear subspace that is dense in $H(\D)$ and an infinite dimensional closed subspace \cite{Charpentier2020}, like the set of universal Taylor series \cite{BayartGrosseErdmanNestoridisPapadimitropoulos2008, Bayartlinearity2005}. There is an extensive literature on the topic of finding linear structures in non-linear sets; we refer to the survey \cite{AronBernalGonzalezPellegrinoSeoaneSepulveda2016}. Despite recent important investigations in the context of linear dynamical systems (see, e.g., \cite{ BayartJFA2019,BayartCostaPapathanasiou2021,BernalGonzalezCalderonMoreno2019,BesConejeroPapathanasiou2018,BesPapathanasiou2020,FalcoGrosse-Erdmann2020} and the recent paper \cite{AlexandreGilmoreGrivaux2023}), it is not known whether the set of universal Taylor series contains an algebra (except the constant functions). We shall see that Theorem \eqref{1} can be used to prove that $\UU_A(\D,\rho)\cup \{\D\ni z\mapsto c\in\C\text{ for some constant }c\in\C\}$ does contain a dense infinitely generated algebra.
	
	The paper is organized as follows: Section \ref{inv-comp} contains the proof of Theorem \eqref{1}, and the application to the algebrability of $\UU_A(\D,\rho)$. Section \ref{4.2} is devoted to the proof of Theorem \eqref{2}. In Section \ref{appl} we provide some further extensions of our main results to a more general notion than Abel universality. In this case, the \textit{origin} $0$ of the radii where we have universal approximation is replaced by a fixed point $w$ in $\mathbb{D}$. We also show that the corresponding classes depend on the choice of $w$ (which is not the case for universal Taylor series if we change the center of expansion). Finally, in Section \ref{residual} we establish the existence of a subset $\UU$ of $\UU_A(\D,\rho)$ which is residual in $H(\D)$ (that is, $\UU$ contains a dense $G_{\delta}$ subset of $H(\D)$) and satisfies the following property: for any $f\in\UU$, any non-constant holomorphic function $g:f(\D)\to \C$, and any automorphism $\Phi$ of $\D$, the function $g\circ f \circ \Phi$ belongs to $\UU$.

	\section{Invariance under composition from the left and algebrability}\label{inv-comp}
	
	In this section, our aim is to prove that all properly defined (non-constant) holomorphic functions preserve Abel universality under composition from the left. Then we will use this result to derive the existence of subalgebras of $H(\D)$, every non-constant element of which is an Abel universal function.
	
	\subsection{Main result}\label{comp_left}
	
	Throughout this paper, the notation $\rho$ stands for an increasing sequence $(r_n)_n$ in $[0,1)$ that tends to $1$ as $n\to \infty$.
	\begin{theorem}\label{left_preserve}
		Let $f\in\UU_A(\D,\rho)$. For any $g:f(\D)\to\C$ non-constant holomorphic function, we have $g\circ f\in\UU_A(\D,\rho)$.
	\end{theorem}
	
	The main tool for the proof is a result similar to a theorem due to Iversen (see \cite{BergweilerEremenko2008} p.245 or \cite{Iversen1914}). We are grateful to Professor Alexandre Eremenko for bringing this theorem to our attention. We will prove this Iversen-type theorem (Theorem \ref{iversen} below) with the help of a consequence of \cite[Theorem 3]{Kaplan1954}, which is an extension of the Gross Star Theorem, and that can be stated as follows:
	\begin{theorem}[Theorem 3 in \cite{Kaplan1954}]\label{gross}
		Let $B$ be a countable closed subset of $\C$. We set $D=\C\setminus B$. Let also $g:D\to\C$ be a non-constant holomorphic function and $w\in g(D)$ any complex number that is not a critical value of $g$. Let $g^{-1}$ be any local inverse of $g$ defined in a neighbourhood of $w$. Then, $g^{-1}$ can be continued indefinitely along almost every ray starting from $w$. More precisely, for almost every $\zeta\in\T$, if $L_{\zeta}$ is the ray starting from $w$ and passing through $w+\zeta$ and $f:[0,\infty)\to L_{\zeta}$, given by $f(t)=w+t\zeta$, is the natural parametrisation of $L_{\zeta}$, then there exists a continuous path $f_0:[0,\infty)\to D$ such that $g\circ f_0\equiv f$. In particular, $f_0\circ f^{-1}$ is a continuous extension of the local inverse $g^{-1}$ of $g$ to the ray $L_\zeta$.
	\end{theorem}
	
	Our Iversen-type theorem reads as follows:
	
	\begin{theorem}\label{iversen}
		Let $B$ be a countable closed subset of $\C$. We set $D=\C\setminus B$. Let also $g:D\to\C$ be a non-constant holomorphic function. Then, for every proper compact subarc $K$ of $\T$, the set of all functions $h\in C(K)$ with the property that there exists $h_0\in C(K)$ such that $g\circ h_0=h$, is dense in $C(K)$.
	\end{theorem}
	\begin{proof}
		Let $K:=\{e^{it}:t\in[\alpha,\beta]\}$, $\alpha,\beta\in[0,2\pi]$ with $\alpha<\beta$, be a proper compact subarc of $\T$ and $h\in C(K)$. Let also $\varepsilon>0$. Then, there exist a finite number of points $t_1,\dots,t_n$ with $\alpha=t_1<t_2<\dots<t_n=\beta$ such that $h(K)\subset\cup_{i=1}^nD\left(h(t_n),\frac{\varepsilon}{4}\right)$. Now, let $w_1\in D\left(h(t_1),\frac{\varepsilon}{4}\right)$ be any point that is not a critical value of $g$. Note that the assumption on $g$ ensures that such a point exists, by Picard's theorem and the fact that $g$ has at most countably many critical values in an open set. Let $g^{-1}$ be any local inverse of $g$ defined in a neighbourhood of $w$. By Theorem \ref{gross}, the function $g^{-1}$ can be continued indefinitely along almost all rays starting from $w_1$. Using the same previous argument, we can then choose a point $w_2\in D\left(h(t_2),\frac{\varepsilon}{4}\right)$ that is not a critical value of $g$, such that $[w_1,w_2]$ is contained in $D\left(h(t_1),\frac{\varepsilon}{4}\right)\cup D\left(h(t_2),\frac{\varepsilon}{4}\right)$ and $g^{-1}$ has a well-defined continuation on $[w_1,w_2]$. More precisely, let $\tilde{h}:\{e^{it}:t\in[t_1,t_2]\}\to[w_1,w_2]$ be the natural parametrisation of $[w_1,w_2]$ with $\tilde{h}(t_1)=w_1$ and $\tilde{h}(t_2)=w_2$. Then there exists a continuous function $h_0$ on $\{e^{it}:t\in[t_1,t_2]\}$ such that $g\circ h_0=\tilde{h}_0$ on $\{e^{it}:t\in[t_1,t_2]\}$.
		
		By repeating this construction, we get the following: there exist points $w_1,\dots,w_n$ such that, for any $k=1,\dots,n$, we have $w_k\in D\left(h(t_k),\frac{\varepsilon}{4}\right)$, $w_k$ is not a critical value of $g$ and $[w_k,w_{k+1}]\subset D\left(h(t_k),\frac{\varepsilon}{4}\right)\cup D\left(h(t_{k+1}),\frac{\varepsilon}{4}\right)$ for $k=1,\dots,n-1$. Also, $g^{-1}$ has a well defined continuation on $\cup_{k=1}^{n-1}[w_k,w_{k+1}]$; that is, for every $k=1,\dots,n-1$, if $\tilde{h}:\{e^{it}:t\in[t_k,t_{k+1}]\}\to[w_k,w_{k+1}]$ is the natural parametrisation of $[w_k,w_{k+1}]$ with $\tilde{h}_0(t_k)=w_k$ and $\tilde{h}_0(t_k)=w_{k+1}$, then there exists a continuous function $h_0$ on $\{e^{it}:t\in[t_k,t_{k+1}]\}$ such that $g\circ h_0=\tilde{h}_0$ on $\{e^{it}:t\in[t_k,t_{k+1}]\}$.
		
		We notice that $h_0$ is a well-defined continuous function on $K$. Also, for any $t\in[\alpha,\beta]$, we have that $t\in[t_k,t_{k+1}]$ for some $k=1,\dots,n-1$ and then $\left|g(h_0(t))-h(t)\right|<\varepsilon$, since $[w_k,w_{k+1}]\subset D\left(h(t_k),\frac{\varepsilon}{4}\right)\cup D\left(h(t_{k+1}),\frac{\varepsilon}{4}\right)$, which completes the proof.
	\end{proof}
	
	We now turn to the proof of Theorem \ref{left_preserve}.
	\begin{proof}[Proof of Theorem \ref{left_preserve}]Let $g:f(\D)\to\C$ be a fixed non-constant holomorphic function, and let $f$ be in $\UU_A(\D,\rho)$.
		By \cite[ Corollary 3.6]{CharpManoMaroI}, the set $f(\D)$ is equal to $\C$ or $\C\setminus\{\alpha\}$ for some $\alpha\in\C$. So Theorem \ref{iversen} applies to $g$ and gives us, for every proper compact subset $K$ of $\T$, the existence of a dense subset $D(K)$ of $C(K)$ such that for every $h\in D(K)$, there exists $h_0\in C(K)$ with $h=g\circ h_0$. Let then $K$ be a proper compact subset of $\T$, $h\in D(K)$ and let $h_0$ be a continuous function on $K$ such that $h=g\circ h_0$ on $K$. Since $f\in\UU_A(\D,\rho)$, there exists an increasing sequence $(\lambda_n)_n$ of integers such that $(f_{r_{\lambda_n}})_n$ converges uniformly to $h_0$ on $K$.
		Now, by compactness, there exists an open set $W\in \C$ such that $h_0(K)\subset W\subset \overline{W}\subset f(\D)$, on which $g$ is well-defined and uniformly continuous. Therefore, $g\circ f_{r_{\lambda_n}}$ is well-defined on $K$ as well, and
		\[
		g\circ f_{r_{\lambda_n}} \to g\circ h_0=h,
		\]
		uniformly on $K$. Since $D(K)$ is dense in $C(K)$, we conclude that $g\circ f\in\UU_A(\D,\rho)$.
	\end{proof}
	We immediately deduce from Theorem \ref{left_preserve} the following corollary, that answers Questions (i) and (ii) mentioned in the introduction. We recall that $\alpha\in \C\cup \{\infty\}$ is an asymptotic value of $f\in H(\D)$ provided that there exists a path $\gamma:[0,1)\to \D$, with $|\gamma(t)|\to 1$ as $t\to 1$, such that $f(\gamma(t))\to \alpha$ as $t\to 1$. By \cite[Corollary 3.6]{CharpManoMaroI}, every $f \in \UU_A(\D,\rho)$ has $\infty$ as an asymptotic value.
	\begin{cor}Let $f\in\UU_A(\D,\rho)$.
		\begin{enumerate}
			\item The function $e^f$ is Abel universal. In particular, there exist zero-free Abel universal functions (we recall that $f(\D)$ is equal to $\C$ or $\C\setminus \{\alpha\}$ for some $\alpha$, by \cite[Corollary 3.6]{CharpManoMaroI}).
			\item If $f$ is zero-free, then for every $\alpha \in \C$, $\alpha + 1/f$ is an Abel universal function. As a consequence, for every $\alpha \in \C$, there exists $f\in \UU_A(\D,\rho)$ that has $\alpha$ as an asymptotic value.
		\end{enumerate}
	\end{cor}
	
	As a consequence of the second point of the previous corollary, note that, on the one hand, the class $\UU_A(\D,\rho)$ is not closed under multiplication. On the other hand, the function $P(f)$ is Abel universal for any $f\in \UU_A(\D,\rho)$ and any non-constant polynomial $P$. In particular, the subalgebra generated by $f$ is contained in $\UU_A(\D,\rho)$, up to the constants. The purpose of the rest of the section is to prove that, in fact, $\UU_A(\D,\rho)$ contains, except the constants, quite \emph{large} subalgebras of $H(\D)$.

	\subsection{Algebrability of $\UU_A(\D,\rho)$}
	Let us start with some definitions.
	\begin{definition}
		Let $X$ be a unital algebra over $\C$. For a subset $B$ of $X$ we denote by $\mathcal{A}(B)$ the subalgebra generated by $B$, that is
		$$\mathcal{A}(B)=\{P(b_1,\dots,b_n):b_1,\dots,b_n\in B,P\in\C[t_1,\dots,t_n],n\in\N\}.$$
		If the set $B$ is algebraically independent, we will say that $B$ freely generates the algebra $\mathcal{A}(B)$.
	\end{definition}
	With the assumptions of the previous definition, we say that a subset $A$ of $X$ is \emph{algebrable} if $A\cup\{c\cdot\mathbf{1}:c\in\C\}$ contains a subalgebra that is not generated by a finite set. In our case, we will use a stronger form of algebrability. We give the appropriate definition introduced in \cite{BartoszewiczGlab2013}.
	\begin{definition}
		Let $X$ be a unital commutative algebra with unit $\mathbf{1}$ over $\C$ that is also a topological space. A subset $A$ of $X$ is \emph{strongly-algebrable} if $A\cup\{c\cdot\mathbf{1}:c\in\C\}$ contains a subalgebra that is freely generated by an infinite set. If the subalgebra is also dense in $X$, we say that $A$ is \emph{densely strongly-algebrable}.
	\end{definition}
	From Theorem \ref{left_preserve}, we get the following:
	\begin{theorem}\label{strong-alg}
		The class $\UU_A(\D,\rho)$ is strongly-algebrable.
	\end{theorem}
	\begin{proof}
		Let $B$ be an infinite set that freely generates $H(\C)$ and fix any $f\in\UU_A(\D,\rho)$. We define the set $\tilde{B}:=\{g\circ f:g\in B\}$. By Theorem \ref{left_preserve}, the algebra generated by $\tilde{B}$ is contained in $\UU_A(\D,\rho)$, except the constants. Moreover, the set $\tilde{B}$ is algebraically independent. Indeed, assume on the contrary that there exist $g_1,\dots,g_N$ and a polynomial $P$ in $N$ variables such that $P(g_1,\dots,g_N)\circ f\equiv0$. Then, by the Identity Principle, we get that $P(g_1,\dots,g_N)\equiv0$, hence $P\equiv0$. This finishes the proof.
	\end{proof}
	We note that the idea of composing a fixed function in our class with an appropriate set of functions has already been used in \cite{GarciaGrecuMaestreSeoaneSepulveda2010}, \cite{AlbuquerqueBernalGonzalezPellegrinoSeoaneSepulveda2014} and \cite{BernalGonzalezLopezSalazarSeoaneSepulveda2020}.
	
	\medskip
	
	To complete the discussion, we shall see that one can actually prove that the set $\UU_A(\D,\rho)$ is densely strongly-algebrable, by using an extension of a criterion given by B\`es and Papathanasiou in \cite{BesPapathanasiou2020}, for the algebrability of the set of hypercyclic vectors for multiplicative operators. Indeed, straightforward modifications to the proof of \cite[Theorem 1.6]{BesPapathanasiou2020} (the details are left to the reader) yields to the following generalisation of their result to general sequences of operators (instead of iterates of a single operator). We refer to \cite{GrosseErdmannPerisManguillot2011} for the relevant concepts from the theory of universality.
	\begin{theorem}\label{alg_crit}
		Let $X,Y$ be separable, commutative unital algebras over the real or complex scalar field $\C$ that are also Fr\'echet spaces. Assume also that $X$ supports a dense freely generated subalgebra. Let also $(T_n)_n$ be a sequence of  multiplicative and continuous operators from $X$ to $Y$ with the property that $(\underbrace{T_n\times\dots\times T_n}_{N\text{ times}})_n$ is topologically transitive for every $N\in\N$. Then, the following are equivalent:
		\begin{enumerate}
			\item The sequence $(T_n)_n$ supports a universal algebra (except the constants).
			\item For each non-constant polynomial $P$ in one variable with coefficients in $\C$, the map $\hat{P}:Y\to Y$ with $f\to P(f)$ has dense range.
			\item For each $N\in\N$ and each non-constant polynomial $P$ in $N$ variables with coefficients in $\C$, the map $\hat{P}:Y^N\to Y$ with $f\to P(f)$ has dense range.
			\item The set of all sequences $(f_j)_j\in X^{\N}$ that freely generate a dense universal algebra (except the constants) is residual in $X^{\N}$.
		\end{enumerate}
	\end{theorem}
	The previous criterion leads us to the following strengthening of Theorem \ref{strong-alg}.
	\begin{theorem}\label{alg-papath}
		The set of all sequences $(f_j)_{j}\in H(\D)^{\N}$ that freely generate a dense subalgebra of $H(\D)$, every non-constant element of which is an Abel universal function, is residual in $H(\D)^{\N}$. In particular, $\UU_A(\D,\rho)$ is densely strongly-algebrable.
	\end{theorem}
	\begin{proof}
		For a proper compact subset $K$ of $\mathbb{T}$ we denote by $\mathcal{U}_{A}^K(\mathbb{D},\rho)$ the set of all holomorphic functions $f$ on $\mathbb{D}$ with the property that $\{  f_{r_{n}} : K \to \mathbb{C}: n\in\mathbb{N} \}$ is dense in $C(K)$. Clearly,
		$$\mathcal{U}_{A}(\mathbb{D},\rho)=\bigcap_{n\in\N}\mathcal{U}_{A}^{K_n}(\mathbb{D},\rho),$$
		where $(K_n)_n$ a sequence of compact proper subsets of $\T$, with the property that for every proper compact subset $K$ of $\T$, there exists $n\in \N$ such that $K\subset K_n$.
		So, by Baire Category Theorem theorem, it is enough to show that, for a fixed proper compact subset $K$ of $\mathbb{T}$, the set of all $(f_j)_j\in H(\D)^{\N}$ that freely generate a dense subalgebra  contained in $\UU_A^K(\D,\rho)$ (except the constants), is residual in $H(\D)^{\N}$.
		
		We will apply Theorem \ref{alg_crit} with $X=H(\D),Y=C(K)$ and $T_n(f)=f_{r_n}$, $f\in H(\D)$, $n\in\N$. Using the standard terminology from the universality theory, the set $\mathcal{U}_{A}^K(\mathbb{D},\rho)$ is the set of all universal elements for $(T_n)_n$. Let $N\in\N$, we will check that $(\underbrace{T_n\times\dots\times T_n}_{N\text{ times}})_n$ is topologically transitive. Indeed, let $L$ be a compact subset of $\D,\varepsilon>0,f_1,\dots,f_N\in H(\D)$ and $h_1,\dots,h_N\in C(K)$. We intend to find $g_1,\dots,g_N\in H(\D)$ with
		$$\sup_{z\in L}|g_i(z)-f_i(z)|<\varepsilon, \ \ \ i=1,\dots,N$$
		and $n_0\in\N$ such that
		$$\sup_{z\in K}|T_{n_0}(g_i)(z)-h_i(z)|<\varepsilon, \ \ \ i=1,\dots,N.$$
		Let $n_0\in\N$ be such that $L\cap r_{n_0}K=\emptyset$, where $r_{n_0}K:=\{r_{n_0}z:\,z\in K\}$. Then, for $i=1,\dots,N$, by Mergelyan's theorem there exists a polynomial $g_i$ such that
		$$\sup_{z\in L}|g_i(z)-f_i(z)|<\varepsilon \quad \text{and} \quad \sup_{z\in K}|g_i(r_{n_0}z)-h_i(z)|<\varepsilon.$$ Then the functions $g_1,\dots,g_N$ have the desired properties.
		
		Thus, we proved that (1) in Theorem \ref{alg_crit} holds, and so (4) is true as well, as desired.
	\end{proof}
	
	\section{Invariance under composition from the right}\label{4.2}
We recall that by $\rho$ we will denote a sequence $(r_n)_n$ in $[0,1)$ that tends to $1$ as $n\to \infty$.	Our aim in this section is to prove the following theorem. 
	
	\begin{theorem}\label{main-thm-invariance-right}
	
	Let $\Phi:\D\to \D$ be holomorphic on $\D$ and continuous on $\overline{\D}$. Then $f\circ \Phi$ belongs to $\UU_A(\D)$ for any $f$ in $\UU_A(\D,\rho)$ if and only if $\Phi$ is a rotation.
	\end{theorem}
	
	It is clear that $\UU_A(\D,\rho)$ is invariant under composition from the right by any rotation. So, using the fact that $\UU_A(\D)=\cup_{\rho}\UU_A(\D,\rho)$, Theorem \ref{main-thm-invariance-right} yields the following:
	
	\begin{cor}\label{cor-noninvariance}For any $\Phi:\D\to \D$ holomorphic and continuous on $\overline{\D}$, the following statements are equivalent:
		\begin{enumerate}
			\item $\UU_A(\D,\rho)$ is invariant under composition from the right by $\Phi$;
			\item $\UU_A(\D)$ is invariant under composition from the right by $\Phi$;
			\item $\Phi$ is a rotation.
		\end{enumerate}
	\end{cor}
	
	The proof of Theorem \ref{main-thm-invariance-right} will be divided into two steps. The first one consists in reducing the problem, using the following proposition.
	
	\begin{prop}\label{prop-Blaschke-to-autom}Let $\Phi:\D\to \D$ be holomorphic and continuous on $\overline{\D}$ and $f\in \UU_A(\D,\rho)$. If $f\circ \Phi \in \UU_A(\D)$, then $\Phi$ is an automorphism.
	\end{prop}
	
	\begin{proof}Let $\Phi$ and $f$ be as in the statement. Let us recall that given any path $\gamma:[0,1) \to \D$ such that $\gamma(r)\to \zeta$ as $r\to 1$ for some $\zeta \in \T$, the set $f(\gamma([0,1))$ is dense in $\C$ (see \cite{Charpentier2020}). Therefore, the continuity of $\Phi$ on $\overline{\D}$ implies that, if $f\circ \Phi \in \UU_A(\D,\rho)$, then $|\Phi(\gamma(r))|\to 1$ as $r\to 1$ for any path $\gamma$ as above. Now, a standard argument shows that a function in $H(\D)$ that enjoys the latter property must be a finite Blaschke product.
		
		Let us assume that $\Phi$ is a finite Blaschke product that is not an automorphism. It is then $l$-valent on $\overline{\D}$ for some integer $l\geq 2$. We claim that for any compact set $L$ in $\D$, there exists $r_0\in [0,1)$ and a compact arc $I$ in $\T$, different from $\T$, such that for any $r\in [r_0,1)$, $\C\setminus \Phi(rI)$ has a bounded connected component $L_r$, contained in $\D$, and such that $L\subset L_r$. Indeed, by continuity of $\Phi$ on $\overline{\D}$ and since $\Phi(\T)=\T$, the set $\Phi(C(0,r))$ is a Jordan curve with at least one point of self-intersection and such that the index of $0$ with respect to it is equal to $l$, that is also the index of $0$ with respect to $\Phi(\T)$. Now the claim follows from the fact that $l\geq 2$ and that $\Phi$ is injective in some neighbourhood of any point of $\T$ (since $\Phi'(\zeta)\neq 0$ for any $\zeta \in \T$).
		
		To conclude, if we assume that $f\circ \Phi$ belongs to $\UU_A(\D)$, then there exists an increasing sequence $(r_n)_n \in [0,1)$, tending to $1$, such that $|f\circ \Phi(r_n\zeta)|\leq 1$ for any $\zeta \in I$. By the maximum modulus principle, this implies that, for any $n$ large enough, $f$ is bounded by $1$ on the set $L_{r_n}$ given by the claim. Since, $L_{r_n}$ contains any compact subset of $\D$ for $n$ large enough, it follows that $f$ is bounded in $\D$, which contradicts the fact that $f$ is in $\UU_A(\D,\rho)$.
	\end{proof}

	By the previous proposition, in order to prove Theorem \ref{main-thm-invariance-right}, we are reduced to proving the following one.
	
	\begin{theorem}\label{thm-noninvariance}Let $\Phi$ be an automorphism of $\D$. Then, the function $f\circ \Phi$ belongs to $\UU_A(\D)$ for any $f$ in $\UU_A(\D,\rho)$, if and only if $\Phi$ is a rotation.
	\end{theorem}
	
	For the proof of Theorem \ref{thm-noninvariance} we will use the following lemma.
	
	\begin{lemma}\label{lemme-conf-not-inv}Let $\Phi$ be an automorphism of $\D$.
		\begin{enumerate}
			\item $\Phi$ is \emph{not} a rotation if and only if for any non-empty open arc $I\subset \T$, there exist $\zeta_1$ and $\zeta_2$ in $I$ such that for any $r\in (0,1)$, $|\Phi(r\zeta_1)|\neq |\Phi(r\zeta_2)|$.
			\item For every $\zeta \in \T$, there exists $r_0\in (0,1)$ such that the function $r\mapsto |\Phi(r\zeta)|$ is increasing on $(r_0,1)$.
		\end{enumerate}
	\end{lemma}
	
	\begin{proof}[Proof of Lemma \ref{lemme-conf-not-inv}] For the first part, we note that if $\Phi$ is a rotation we obviously have $|\Phi(r\zeta_1)|= |\Phi(r\zeta_2)|$ for all $r\in (0,1)$ and $\zeta_1$, $\zeta_2\in\T$ . Let us check the non-trivial direction. Without loss of generality we can suppose that $\Phi$ is a M\"obius transformation, namely $\Phi=\Phi_a$ for some $a\in \D$, where
		\begin{equation}\label{Mobius}
			\Phi_a(z)=\frac{a-z}{1-\bar{a}z}.
		\end{equation}
		Let $r\in (0,1)$ and $I$ an open arc in $\T$. Using the formula
		\begin{equation}\label{eq-lemma-not-inv-conf-1}
			1-|\Phi(z)|^2=\frac{(1-|a|^2)(1-|z|^2)}{|1-\bar{a}z|^2},
		\end{equation}
		we easily get that for any $\zeta_1,\zeta_2\in \T$, the equality $1-|\Phi(r\zeta_1)|^2=1-|\Phi(r\zeta_2)|^2$ is equivalent to
		\begin{equation}\label{eq-lemma-not-inv-conf-2}
			\Re{(\bar{a}\zeta_1)}=\Re{(\bar{a}\zeta_2)},
		\end{equation}
		which is also equivalent to $1-|\Phi(R\zeta_1)|^2=1-|\Phi(R\zeta_2)|^2$ for any $R\in (0,1)$. Now, it is clear that \eqref{eq-lemma-not-inv-conf-2} is true for any $\zeta_1,\zeta_2$ in the non-empty open arc $I$ if and only if $a=0$, i.e. $\Phi$ is a rotation.
		
		The second part can be easily seen by using \eqref{eq-lemma-not-inv-conf-1} and studying the function $r\mapsto \frac{1-r^2}{r^2|a|^2-2r\Re{(\bar{a}\zeta)}+1}$.
	\end{proof}
	
	Now we will use the above lemma to prove Theorem \ref{thm-noninvariance}.
	
	\begin{proof}[Proof of Theorem \ref{thm-noninvariance}]We have already noticed that, if $f$ is any function in $\UU_A(\D,\rho)$ and $\Phi$ is a rotation, then $f\circ \Phi$ belongs to $\UU_A(\D,\rho)$. So it is enough to show that if $\Phi$ is any M\"obius transformation that is not a rotation, then there exists $f\in \UU_A(\D,\rho)$ such that $f\circ \Phi$ does not belong to $\UU_A(\D)$. Let us then fix $\Phi=\Phi_a$ with $a\neq 0$ (see \eqref{Mobius} for the definition of $\Phi_a$). Since $\Phi^{-1}=\Phi$, we shall always denote also by $\Phi$ the function $\Phi^{-1}$. 
		
		By assumption and Lemma \ref{lemme-conf-not-inv}, there exist $\zeta _1\neq \zeta_2$ in $\T$ such that for any $r\in (0,1)$, $|\Phi(r\zeta_1)|\neq |\Phi(r\zeta_2)|$. Let $r_{-1}\in (0,1)$ be such that both functions $r\mapsto |\Phi(r\zeta_1)|$ and $r\mapsto |\Phi(r\zeta_2)|$ are increasing on $(r_{-1},1)$ (using (2) of Lemma \ref{lemme-conf-not-inv}). Note that for any $r,r'\in (r_{-1},1)$, $\Phi(r\zeta_1)\neq \Phi(r'\zeta_2)$, by injectivity of $\Phi$ on $\overline{\D}$. 
		
		Let us write $\rho=(r_n)_n$. Without loss of generality we may and shall assume that $r_n\in (r_{-1},1)$ for any $n\in \N$. By all the previous, together with the continuity of $\Phi$ and the fact that $|\Phi(r\zeta)|\to 1$ as $r\to 1$ for any $\zeta \in \T$, the sets $\cup_n\{R\in (0,1):\,|\Phi(R\zeta_1)|=r_n\}$ and $\cup_n\{R\in (0,1):\,|\Phi(R\zeta_2)|=r_n\}$ are disjoint and can be enumerated as, respectively, two increasing sequences $(R_n^1)_n$ and $(R_n^2)_n$ in $(0,1)$, tending to $1$ as $n\to \infty$. Then we have $|\Phi(R_n^i\zeta_i)|=r_n$ for any $n$ and any $i\in \{1,2\}$. For technical reasons, we will also make use of the two increasing enumerations $(s_n^1)_{n\geq -1}$ and $(s_n^2)_{n\geq -1}$ in $(0,1)$ of the sets $\cup_{n\geq -1}\{s\in (0,1):\,|\Phi(s\zeta_1)|=(r_n+r_{n+1})/2\}$ and $\cup_{n\geq -1}\{s\in (0,1):\,|\Phi(s\zeta_2)|=(r_n+r_{n+1})/2\}$, respectively. Again, the lastly considered four sets are disjoint, and we have $R_n^i<s_n^i<R_{n+1}^i$ for any $n\geq -1$ and any $i\in \{1,2\}$. For $n\geq -1$, we denote $L_n=\overline{D(0,r_{n})}$.
		
		\medskip
		
		Finally, let us fix technical notations that we will used in the construction below:
		\begin{itemize}
			\item $(\varepsilon_n)_n$ is a sequence of positive real numbers decreasing to $0$ with $\sum_{n=1}^{\infty}\varepsilon_n\leq 1/2$;
			\item $(\vp_n)_n$ is an enumeration of the polynomials with coefficients in $\Q +i\Q$;
			\item $(K_n)_n$ is a sequence of compact subsets of $\T$, different from $\T$, such that for any compact set $K\subset \T$ different from $\T$, there exists $n\in \N$ such that $K\subset K_n$;
			\item $\alpha,\beta: \N\to \N$ are two functions such that for any $l,m\in \N$, there exists infinitely many $n\in \N$ for which $(\alpha(n),\beta(n))=(l,m)$.
		\end{itemize}
		
		We now proceed to the construction of the desired function. To do so, we will build by induction a sequence $(P_n)_n$ of polynomials with specific approximation properties. We set $P_0=0$. Assume now that $P_0,\ldots,P_{n-1}$ have been defined. To specify $P_n$ we distinguish three cases:
		
		\smallskip
		
		\textbf{Case I:} if $r_nK_{\beta(n)}\cap \Phi([r_{-1},1)\zeta_i)=\emptyset$ for each $i\in \{1,2\}$.
		
		We choose $P_n$ by applying Mergelyan's theorem in order to have for any $i\in \{1,2\}$,
		\begin{enumerate}[(a)]
			\item \label{casea} $\sup_{z\in L_{n-1} \cup \Phi([r_{-1},1]\zeta_i)}|P_n(z)|\leq \e_n$;
			\item \label{caseb-1} $\sup_{z\in r_nK_{\beta(n)}}|P_n(z)+\left(\sum_{k=0}^{n-1}P_k(z)-\vp_{\alpha(n)}(z)\right)|\leq \varepsilon_n$.
		\end{enumerate}
		
		Note that this choice is possible by assumption and since the set $\cup_{i\in \{1,2\}}\Phi([r_{-1},1]\zeta_i) \cup L_{n-1} \cup r_nK_{\beta(n)}$ is compact and has connected complement.
		
		\smallskip
		
		\textbf{Case II:} if $r_nK_{\beta(n)}\cap \Phi([r_{-1},1)\zeta_i)=\emptyset$ and $r_nK_{\beta(n)}\cap \Phi([r_{-1},1)\zeta_j)\neq\emptyset$ for some $i,j\in \{1,2\}$.
		
		Since both cases $(i,j)=(1,2)$ and $(i,j)=(2,1)$ are treated in the same way, we assume that $r_nK_{\beta(n)}\cap \Phi([r_{-1},1)\zeta_2)=\emptyset$. Let $\psi_n(z)$ be any function continuous on the curve $\Phi([r_{-1},1]\zeta_1)$ that is equal to $0$ on $\Phi([r_{-1},s_{n-1}^1]\zeta_1\cup [s_{n}^1,1]\zeta_1)$ and to $\sum_{k=0}^nP_k(z)-\vp_{\alpha(n)}(z)$ at $z=\Phi(R_n^1\zeta_1)$. Then we define the function
		\begin{equation*}
			\vp_n^{\text{II}}(z)=\left\{\begin{array}{ll}
				\sum_{k=0}^{n-1}P_k(z)-\vp_{\alpha(n)}(z) & \text{if }z\in r_nK_{\beta(n)}\\
				\psi_n(z) & \text{if }z\in \Phi([r_{-1},1]\zeta_1).
			\end{array}\right.
		\end{equation*}
		
		Note that $\vp_n^{\text{II}}(z)$ is continuous on $r_nK_{\beta(n)}\cup \Phi([r_{-1},1]\zeta_1)$ and holomorphic in its interior. Moreover, by assumption, the set $\cup_{i\in \{1,2\}}\Phi([r_{-1},1]\zeta_i) \cup L_{n-1} \cup r_nK_{\beta(n)}$ has connected complement. Then we choose $P_n$ by applying Mergelyan's theorem in order to have
		\begin{enumerate}[(a)]
			\setcounter{enumi}{2}
			\item \label{caseb}$\sup_{z\in L_{n-1} \cup \Phi([r_{-1},1]\zeta_2)}|P_n(z)|\leq \e_n$;
			\item \label{cased}$\sup_{z\in r_nK_{\beta(n)}\cup \Phi([r_{-1},1]\zeta_1)}|P_n(z)+ \vp_{n}^{\text{II}}|\leq \varepsilon_n$.
		\end{enumerate}
		
		\smallskip
		
		\textbf{Case III:} if $r_nK_{\beta(n)}\cap \Phi([r_{-1},1)\zeta_i)\neq \emptyset$ for each $i\in \{1,2\}$. There are two subcases that correspond to $R_n^1<R_n^2$ and $R_n^1>R_n^2$. Since both are treated similarly, we only consider the case where $R_n^1<R_n^2$, i.e. $|\Phi(R_n^2\zeta_1)|>r_n$ and $|\Phi(R_n^1\zeta_2)|<r_n$. By continuity, there exists $\eta_n>0$ such that the following inequalities hold:
		\begin{multline*}
			\max\left(|\Phi(s_{n-1}^2\zeta_2)|,|\Phi((R_n^1+\eta_n)\zeta_2)|\right) < |\Phi((R_n^2-\eta_n)\zeta_2)| < r_n \\ < |\Phi((R_n^1+\eta_n)\zeta_1)| < \min\left(|\Phi(s_{n}^1\zeta_1)|,|\Phi((R_n^2-\eta_n)\zeta_1)|\right)
		\end{multline*}
		and
		\[
		|\Phi((R_n^1-\eta_n)\zeta_1|\geq |\Phi(s_{n-1}^1\zeta_1)|\quad \text{and} \quad |\Phi((R_n^2+\eta_n)\zeta_2)|\leq |\Phi(s_n^2\zeta_2)|.
		\]
		Now consider any function $\kappa_n$ continuous on the disjoint union of the two curves $\Phi([r_{-1},1]\zeta_1)$ and $\Phi([r_{-1},1]\zeta_2)$ such that
		\begin{itemize}
			\item $\kappa_n(z)=0$ if $z\in \Phi([r_{-1},(R_n^2-\eta_n)]\zeta_2)\cup \Phi((R_n^2+\eta_n),1]\zeta _2)$ and if $z\in \Phi([r_{-1},(R_n^1-\eta_n)]\zeta_1)\cup \Phi((R_n^1+\eta_n),1]\zeta _1)$;
			\item $\kappa_n(z)$ is equal to $\sum_{k=0}^{n-1}P_k(z)-\vp_{\alpha(n)}(z)$ at $z=\Phi(R_n^1\zeta_1)$ and $z=\Phi(R_n^2\zeta_2)$.
		\end{itemize}
		Finally, we define the function
		\begin{equation*}
			\vp_n^{\text{III}}(z)=\left\{\begin{array}{ll}
				\sum_{k=0}^{n-1}P_k(z)-\vp_{\alpha(n)}(z) & \text{if }z\in r_nK_{\beta(n)}\\
				\kappa_n(z) & \text{if }z\in \Phi([r_{-1},1]\zeta_1)\cup \Phi([r_{-1},1]\zeta_2).
			\end{array}\right.
		\end{equation*}
		
		By definition $\vp_n^{\text{III}}(z)$ is continuous on the compact set $r_nK_{\beta(n)}\cup \Phi([r_{-1},1]\zeta_1)\cup \Phi([r_{-1},1]\zeta_2)$ and holomorphic in its interior, hence also on the smaller compact set
		\[
		G_n:=r_nK_{\beta(n)}\cup \Phi([s_{n-1}^1,1]\zeta_1)\cup \Phi([r_{-1},1]\zeta_2).
		\]
		Note that the $L_{n-1}\cup G_n$ is compact and has connected complement, since $\Phi(s_{n-1}^1\zeta_1)>r_{n-1}$ by definition. Thus we can choose $P_n$ by applying Mergelyan's theorem in such a way that
		\begin{enumerate}[(a)]
			\setcounter{enumi}{4}
			\item \label{casec} $\sup_{z\in L_{n-1}}|P_n(z)|\leq \e_n$;
			\item \label{casef}$\sup_{z\in G_n}|P_n(z)+\vp_{n}^{\text{III}}(z)|\leq \varepsilon_n$.
		\end{enumerate}
		
		This concludes the induction. We set $f=\sum _n P_n$ which defines a function in $H(\D)$ (by (\hyperref[casea]{a}), (\hyperref[caseb]{c}) and (\hyperref[casec]{e})). Let us first check that $f\in \UU_A(\D,\rho)$. Fix $n,m\in \N$ and let $(n_l)_l\subset \N$ be such that for any $l\in \N$, $\alpha(n_l)=n$ and $\beta(n_l)=m$. By (\hyperref[caseb-1]{b}), (\hyperref[cased]{d}) or (\hyperref[casef]{f}) we have, for any $z \in K_m$,
		\begin{equation}\label{eq-Abel-univ-Sn-infty}
			\left|f(r_{n_l}z)-\vp_n(z)\right| \leq \left|\sum _{n=0}^{n_l}P_n(z) -\vp_n(z)\right| + \sum_{n>n_l} |P_n(z)| \leq \e_{n_l} + \sum_{n>n_l}\e_{n} \to 0
		\end{equation}
		as $l\to \infty$.
		To see that $f\circ \Phi$ does not belong to $\UU_A(\D)$, it is enough to see that the construction yields, for any $r\in [r_{-1},1)$,
		\[
		\min\left(|f\circ \Phi(r\zeta_1)|,|f\circ \Phi(r\zeta_2)|\right) \leq \sum_n \e_n.
		\]
		In particular, it is not possible to uniformly approximate on $\{\zeta_1,\zeta_2\}$ any constant greater than $\sum_n \e_n$ by some dilates $(f\circ \Phi)_r$, $r\in [r_{-1},1)$, of $f\circ \Phi$.

	\end{proof}
	
	\begin{remark}{\rm If $\Phi$ is not continuous on $\overline{\D}$, it is not clear what necessary and sufficient condition $\Phi$ should satisfy for the conclusion of the previous corollary to hold.}
	\end{remark}
	
	\section{Further extensions and applications}\label{appl}
	
	In \cite{Charpentier2020}, a notion slightly more general than that of Abel universal functions was introduced. As usual, we denote by $\rho=(r_n)_n$ an increasing sequence in $[0,1)$ that converges to $1$ as $n\to \infty$.
	
	\begin{definition}{\rm Let $\rho$ and $w\in \D$ be fixed. We will say that $f\in H(\D)$ belongs to the class $\UU_A^w(\D,\rho)$ of \textit{$w$-Abel universal functions} if, for any proper compact subset $K$ of $\T$, and any function $\vp \in C(K)$,
			\[
			\sup_{\zeta \in K}\left|f(w+r_n(\zeta -w))-\vp(\zeta)\right|\to 0\quad \text{as }n\to \infty.
			\]
		}
	\end{definition}
	We set
	\[
	\mathcal{U}_{A}^w(\D)=\bigcup_{\rho}\UU_A^w(\D,\rho),
	\]
	where $\rho$ varies over the set of all increasing sequences in $[0,1)$ tending to $1$.
	
	Note that if $w=0$, then $\UU_A^w(\D,\rho)$ coincides with $\UU_A(\D,\rho)$, and $\mathcal{U}_{A}^w(\D)$ with $\UU_A(\D)$. Roughly speaking, the definition of a $w$-Abel universal function is that of an Abel universal function for which the \textit{origin} $0$ of the usual radii is moved to another point $w$ in $\D$. It was proven in \cite{Charpentier2020} that the set
	\[
	\bigcap_{w\in \D}\UU_A(\D,\rho,w)
	\]
	is a residual subset of $H(\D)$.
	
	Since $\UU_A(\D)$ does not coincide with any of the sets $\UU_A(\D,\rho)$ (see \cite{CharpentierMouze2022}), and in view of the result of the previous parts, it is natural to ask the following:
	\begin{questionsss}\label{Questions-1-and-2}$\quad$
		\begin{enumerate}
			\item \label{quest-1-fin}Do the classes $\UU_A^w(\D,\rho)$ depend on the choice of the ``origin'' $w$?
			\item \label{quest-2-fin}How do Theorems \ref{left_preserve} and \ref{main-thm-invariance-right} extend to the classes $\UU_A^w(\D,\rho)$ and $\UU_A^w(\D)$ for $w\in \D$ different from $0$.
		\end{enumerate}
	\end{questionsss}
	
	First of all, we underline that, up to superfluous modifications that are left to the reader, the proof of Theorem \ref{left_preserve} works for $\UU_A^w(\D)$, for any $w\in \D$. Thus we have:
	\begin{theorem}\label{left_preserve_w}
		Let $f\in\UU_A^w(\D,\rho)$. For any non-constant holomorphic function $g:f(\D)\to\C$, the function $g\circ f$ belongs to $\UU_A^w(\D,\rho)$.
	\end{theorem}
	
	For Question \ref{quest-1-fin} and the second part of Question \ref{quest-2-fin}, we shall see that the main ideas of the proof of Theorem \ref{main-thm-invariance-right} are general enough to give complete answers.
	
	\medskip
	
	Let $\Psi$ be a homeomorphism from $\overline{\D}$ onto $\overline{\D}$ and let $(P_{\Psi})$ be the property that is satisfied by a function $f$ in $H(\D)$ whenever, for every proper compact set $K\subset \T$, the set $\{(f\circ \Psi)_{r}:K\to \C;\, r\in [0,1)\}$ is dense in $C(K)$. Analogously, if $\rho=(r_n)_n\subset[0,1)$, tending to $1$ as $n\to \infty$ is given, then a function $f$ in $H(\D)$ satisfies the property $(P_{\Psi}^{\rho})$ whenever the set $\{(f\circ \Psi)_{r_n}:K\to \C;\, n\in \N\}$ is dense in $C(K)$.
	
	Note that $f$ belongs to $\UU_A^w(\D)$ (respectively $\UU_A^w(\D,\rho)$) if and only if it has $(P_{\Psi_w})$ (respectively $(P_{\Psi_w}^{\rho})$), with
	\begin{equation}\label{eq-def-psi-w}
		\Psi_w(re^{i\theta})=w+r(e^{i\theta}-w), \quad re^{i\theta} \in \D.
	\end{equation}
	Similarly, if $f\in H(\D)$ and $\Phi$ is an automorphism of the disc, then $f\circ \Phi$ belongs to $\UU_A(\D)$ (respectively $\UU_A(\D,\rho)$) if and only if $f$ satisfies the property $(P_{\Phi})$ (respectively $(P_{\Phi}^{\rho})$).
	
	\medskip
	
	Now, the crucial technical point that made the construction in the proof of Theorem \ref{thm-noninvariance} possible was the fact that if $\Phi$ is an automorphism of $\D$ that is not a rotation, then there exist $r_0\in [0,1)$ and $\zeta_1,\zeta_2\in \T$ such that $|\Phi(r\zeta_1)|\neq |\Phi(r\zeta_2)|$ for any $r\in [r_0,1)$. With this condition, it was possible to construct a function $f$ in $\UU_A(\D,\rho)$ such that the minimum of $|f\circ \Phi(r\zeta_1)|$ and $|f\circ \Phi(r\zeta_2)|$ is less than a constant, for any $r$ large enough. Up to minor technical details, it is not difficult to see that this condition can be replaced by the slightly weaker one: there exist $r_0\in [0,1)$ and $\zeta_1,\zeta_2,\zeta_3\in \T$ such that for any $r\in [r_0,1)$, the points $\Phi(r\zeta_1)$, $\Phi(r\zeta_2)$ and $\Phi(r\zeta_3)$ do not belong to the same circle centred at $0$ (\textit{i.e.}, there exists $i,j\in \{1,2,3\}$ such that $|\Phi(r\zeta_i)|\neq |\Phi(r\zeta_j)|$). Under this condition, we can build a function $f\in \UU_A(\D,\rho)$ such that for any $r\in [r_0,1)$, the quantity
	\[
	\min\left(|f\circ \Phi(r\zeta_1)|,|f\circ \Phi(r\zeta_2)|,|f\circ \Phi(r\zeta_3)|\right)
	\]
	is less than a finite constant independent of $r$. Thus the function $f\circ \Phi$ cannot belong to $\UU_A(\D)$.
	
	More generally, if $\Psi_1$ and $\Psi_2$ are two homeomorphisms from $\overline{\D}$ onto $\overline{\D}$, the same type of construction can be made in order to prove the following proposition:
	
	\begin{prop}\label{prop-general-invariance-right}Let $\Psi_1$ and $\Psi_2$ be two homeomorphisms from $\overline{\D}$ onto $\overline{\D}$. If there exist $r_0\in [0,1)$ and $\zeta_1,\zeta_2,\zeta_3\in \T$ such that for any $r\in [r_0,1)$, there is no $R\in [r_0,1)$ such that the points $\Psi_2(r\zeta_1)$, $\Psi_2(r\zeta_2)$ and $\Psi_2(r\zeta_3)$ all belong to the same closed curve $\Psi_1(C(0,R))$, then there exists $f\in H(\D)$ that satisfy $(P^{\rho}_{\Psi_1})$, but not $(P_{\Psi_2})$.
	\end{prop}
	
	The proof of this proposition can be written following the same lines as that of Theorem \ref{thm-noninvariance}, up to technical details, and so is omitted. Now, we will explain how this proposition can be used in order to answer the above Questions \ref{quest-1-fin} and \ref{quest-2-fin}.
	
	\subsection{Answer to Question \ref{quest-1-fin}}
	Let $w_1\neq w_2$ in $\D$. To answer Question \ref{quest-1-fin}, one can directly apply Proposition \ref{prop-general-invariance-right} to the homeomorphisms $\Psi_{w_1}$ and $\Psi_{w_2}$ defined by \eqref{eq-def-psi-w}. Indeed, since $\Psi_1(C(0,R))$ and $\Psi_2(C(0,r))$ are different circles for any values of $r,R\in [0,1)$, and since there is only one circle passing through $3$ pairwise distinct points, it is true that, for any $3$-tuple $\zeta_1$, $\zeta_2$ and $\zeta_3$ in $\T$, the points $\Psi_2(r\zeta_1)$, $\Psi_2(r\zeta_2)$ and $\Psi_2(r\zeta_3)$ cannot all belong to $\Psi_1(C(0,R))$, whatever $R\in [0,1)$. Thus, by Proposition \ref{prop-general-invariance-right} we get that all classes introduced above are different:
	
	\begin{prop} Let $\rho_1,\rho_2$ be two increasing sequences in $[0,1)$ tending to $1$ and let $w_1\neq w_2$ in $\D$. Then
		\[
		\mathcal{U}_A^{w_1}(\D)\setminus \mathcal{U}_A^{w_2}(\D) \neq \emptyset\quad \text{and}\quad \UU_A^{w_1}(\D,\rho_1)\setminus \UU_A^{w_2}(\D,\rho_2) \neq \emptyset.
		\]
	\end{prop}
	
	\begin{remark}{\rm One might be tempted to see a contrast with the fact that the universality of the partial sums of universal Taylor series on $\D$ does not depend on the choice of the center of their Taylor expansions, see \cite{MelasNestoridis2001}.
		}
	\end{remark}
	
	\subsection{Anwser to Question \ref{quest-2-fin}}
	
	Let now $w\in \D$ be given, and let $\Phi:\D \to \D$ be holomorphic and continuous on $\overline{\D}$. First of all, the same proof as that of Proposition \ref{prop-Blaschke-to-autom} makes it clear that if $f\circ \Phi$ belongs to $\UU_A^w(\D)$ for any $f\in \UU_A^w(\D,\rho)$, then $\Phi$ is an automorphism. Thus we may and shall assume that $\Phi$ is an automorphism of $\D$.
	
	Now, Proposition \ref{prop-general-invariance-right} tells us that if $\Phi$ is an automorphism of $\D$ such that $f\circ \Phi \in \UU_A^w(\D)$ for any $f\in \UU_A^w(\D,\rho)$, then there exist at least two (in fact infinitely many) distinct numbers $r_1,r_2$ in $[0,1)$ and two (again, in fact infinitely many) $3$-tuples $(\zeta_1,\zeta_2,\zeta_3)$ and $(\xi_1,\xi_2,\xi_3)$ of pairwise distinct points in $\T^3$, such that there exist $R_1,R_2\in [0,1)$ for which
	\begin{itemize}
		\item the points $\Phi\circ \Psi_w(r_1\zeta_1),\Phi\circ \Psi_w(r_1\zeta_2),\Phi\circ \Psi_w(r_1\zeta_3)$ belong to $\Psi_w(C(0,R_1))$;
		\item the points $\Phi\circ \Psi_w(r_2\xi_1),\Phi\circ \Psi_w(r_2\xi_2),\Phi\circ \Psi_w(r_2\xi_3)$ belong to $\Psi_w(C(0,R_2))$.
	\end{itemize}
	Observe that the pairwise distinct points $r_1\zeta_i$ (respectively $r_2\xi_i$), $i\in \{1,2,3\}$, belong to the circle $C(0,r_1)$ (respectively $C(0,r_2)$), and that $\Psi_w(C(0,r))$ is the circle $C((1-r)w,r)$. Since automorphisms of $\D$ map circles onto circles and since there is only one circle passing through $3$ pairwise distinct points, we deduce from the previous that the property ``$f\circ \Phi \in \UU_A^w(\D)$ for any $f\in \UU_A^w(\D,\rho)$'' holds only if, for at least two pairs $(r_1,r_1'),(r_2,r_2')\in (0,1)^2$, with $r_i\neq r_j$ and $r_i'\neq r_j'$, $\Phi$ maps the circle $C((1-r_i)w,r_i)$ onto $C((1-r'_i)w,r'_i)$, $i=1,2$.
	
	In the case $w=0$, the proof of Lemma \ref{lemme-conf-not-inv} tells us that the latter condition can only happen if $\Phi$ is a rotation (or it can also be recalled as a well-known fact that the only automorphisms of $\D$ that map a circle centred at $0$ onto a circle centred at $0$ are the rotations). The following lemma is a generalization of this fact, and tells us in particular that, if $w\neq 0$, this condition holds only if $\Phi$ is the identity map. For simplicity, we will use in the sequel the notation $C_w(r)=C((1-r)w,r)$.
	
	\begin{lemma}\label{lemma-noninv-w}Let $\Phi$ be an automorphism of $\D$. Let $(r_1,r_1'),(r_2,r_2')\in (0,1)^2$, with $r_i\neq r_j$ and $r_i'\neq r_j'$, and assume that $\Phi$ maps $C_w(r_i)$ onto $C_w(r_i')$ for $i=1,2$. Then $r_i=r_i'$ for $i=1,2$ and, either $w=0$, in which case $\Phi$ is a rotation, or $w\neq0$ and $\Phi(z)=z$ for any $z\in \D$.
	\end{lemma}
	
	\begin{proof}
		Let $\Phi$ be an automorphism of $\D$ satisfying the condition of the lemma. Then, necessarily, $\Phi$ maps $\overline{D((1-r_i)w,r_i)}$ onto $\overline{D((1-r'_i)w,r'_i)}$, $i=1,2$. Now, it is easily seen that $D((1-r)w,r)$ is contained in  $D((1-r')w,r')$ whenever $r\leq r'$. As a consequence, $\Phi$ must have a fixed point in one of the above discs, and so in $\D$. Let us denote by $a$ this fixed point. Let us also denote by $\vp_a$ an automorphism of $\D$ that maps $0$ to $a$. Then $\Psi:=\vp_a^{-1}\circ \Phi \circ \vp_a$ is an automorphism of $\D$ that fixes $0$, hence $\Psi$ is a rotation. Moreover, since $\Phi$ maps $C_w(r_i)$ onto $C_w(r_i')$ for $i=1,2$, we get
		\[
		\Psi\left(\vp_a^{-1}\left(C_w(r_i)\right)\right)=\vp_a^{-1}\left(C_w(r'_i)\right).
		\]
		Moreover, using again the fact that $D((1-r)w,r)$ is contained in  $D((1-r')w,r')$ whenever $r\leq r'$, and since automorphisms of $\D$ map circles onto circles, $\vp_a^{-1}\left(C_w(r_i)\right)$ and $\vp_a^{-1}\left(C_w(r'_i)\right)$ are circles and one of them is contained in the open disc bounded by the other. Since $\Psi$ is a rotation, this implies that $r_i=r_i'$ and that, either $\Psi$ is the identity map, or that $\vp_a^{-1}\left(C_w(r_i)\right)$ is a circle $C(0,R_i)$ centred at $0$, for some $R_i\in (0,1)$, $i=1,2$, with $R_1\neq R_2$. If $\Psi$ is the identity map, then so is $\Phi$. Let us then assume that $\Psi$ is a rotation that is not the identity map.
		
		Then $\vp_a\left(C(0,R_i)\right)=C_w(r_i)$ for $i=1,2$. Let $R\in(0,1)$ be fixed. By Schwartz-Pick lemma, $\vp_a(C(0,R))$ is the pseudo-hyberbolic disc centred at $a$ with radius $R$ (for the pseudo-hyperbolic distance), or equivalently the hyperbolic disc centred at $a$ with radius $2\tanh^{-1}(R)$ (with respect to the hyperbolic distance), see for e.g. \cite[Theorem 2.2]{BeardonMinda2007}. This disc is a Euclidean one with center $c$ and (euclidean) radius $r$ given by the formulae (\cite[Exercise 3, Page 16]{BeardonMinda2007})
		\[
		c=\frac{a\left(1-R^2\right)}{1-|a|^2R^2}\quad\text{and}\quad r=\frac{\left(1-|a|^2\right)R}{1-|a|^2R^2}.
		\]
		Thus, given $r\in (0,1)$, the equality $\vp_a\left(C(0,R)\right)=C_w(r)$ for some $R\in (0,1)$ imposes the conditions
		\begin{equation}\label{eq-Uw-inv}
			(1-r)w=\frac{a\left(1-R^2\right)}{1-|a|^2R^2}\quad\text{and}\quad r=\frac{\left(1-|a|^2\right)R}{1-|a|^2R^2}.
		\end{equation}
		Replacing the value of $r$ given by the second equality in the first one, a direct computation leads to the conditions $a=0$ or
		\[
		R=\frac{w-a}{a-|a|^2w}.
		\]
		If $a\neq 0$ then the above condition imposes that $R_i=\frac{w-a}{a-|a|^2w}$, $i=1,2$, which contradicts the fact that $R_1\neq R_2$. Therefore $a=0$, that is equivalent to $w=0$ by \eqref{eq-Uw-inv}. Now $w=0$ implies that $\Phi$ is rotation by Lemma \ref{lemme-conf-not-inv}, which concludes the proof of the lemma.
	\end{proof}
	
	Altogether, we have proved the following complement to Theorem \ref{thm-noninvariance}:
	
	\begin{theorem}\label{thm-w-general-noninvariance}Let $w\in \D\setminus \{0\}$ and let $\Phi:\D \to \D$ be holomorphic and continuous on $\overline{\D}$. Then $f\circ \Phi$ belongs to $\UU_A^w(\D)$ for any $f \in \UU_A^w(\D,\rho)$ if and only if $\Phi$ is the identity map.
	\end{theorem}
	
	The explanations that lead us to this result suggest the following remark.
	
	\begin{remark}{\rm In view of the above explanations, the definition of the class $\UU^w_A(\D,\rho)$ may not seem natural. Rather, it might be more natural to define the class $\VV_A^w(\D,\rho)$ of those $f\in H(\D)$ such that $f\circ \vp_w \in \UU_A(\D,\rho)$, where $\vp_w$ is an automorphism of $\D$ that maps $0$ to $w$. In this case, $\vp_w$ sends circles centered at $0$ to hyperbolic circles centered at $w$ and it is clear from the proof of Theorem \ref{thm-w-general-noninvariance} that if $\Phi:\D\to \D$ is holomorphic and continuous on $\overline{\D}$, then $f\circ \Phi \in \VV_A^w(\D,\rho)$ for any $f\in \VV_A^w(\D,\rho)$ if and only if $\vp_w^{-1}\circ \Phi\circ \vp_w$ is a rotation.
		}
	\end{remark}

	\section{A residual set of Abel universal functions that is invariant under composition by any automorphism}\label{residual}
	
	We conclude the paper by establishing the existence of a class of functions that is residual in $H(\D)$, contained in $\UU_A(\D,\rho)$, and that is invariant under composition from the right by any automorphism of $\D$.
	
	\begin{theorem}\label{thm-all-auto}Let $\rho=(r_n)_n$ be an increasing sequence converging to $1$ as $n\to \infty$. There exists a residual set of functions $f\in H(\D)$ such that $f\circ\Phi\in\UU_A(\D,\rho)$ for every automorphism $\Phi$ of the unit disc.
	\end{theorem}
	
	\begin{proof}According, for e.g., to Theorem \ref{thm-noninvariance}, it is enough to prove the existence of a residual set of functions $f\in H(\D)$ such that $f\circ\Phi\in\UU_A(\D,\rho)$ for every automorphism of the form $\Phi_w(z)=\frac{z+w}{1+\overline{w}z}$, $w\in \D$.
		
		Let $(\vp_n)_n$ be an enumeration of the polynomials with coefficients in $\Q +i\Q$ and let $(K_n)_n$ be a sequence of compact subsets of $\T$, different from $\T$, such that for any compact set $K\subset \T$ different from $\T$, there exists $n\in \N$ such that $K\subset K_n$.
		
		We claim that, for any fixed integers $k,m,n$, there exists a sequence $(w_l(k,m,n),\delta_l(k,m,n))_l$ in $\D\times (0,+\infty)$ such that the following conditions hold:
		\begin{enumerate}[(i)]
			\item for every $w\in \D$, there exists $l \in \N$ such that $w\in D(w_l(k,m,n),\delta_l(k,m,n))$;
			\item for every $l\in \N$, the set
			\[
			F(k,m,n,l):=\bigcup_{w\in \overline{D(w_l(k,m,n),\delta_l(k,m,n))}}\Phi_w(r_kK_n).
			\]
			is contained in $\D$ and has connected complement in $\C$;
			\item for every $l\in \N$ and $\tau \in D(w_l(k,m,n),\delta_l(k,m,n))$,
			\[
			\sup_{z\in \overline{D}}|\vp_m(\Phi_{w_l(k,m,n)}^{-1}(\Phi_{\tau}(r_kz)))-\vp_m(r_kz)|<\frac{1}{m}.
			\]
		\end{enumerate}
		
		Let us check the claim first. We fix $k,m,n\in \N$. If $L$ is a compact set in $\D$ and $w\in L$ is fixed, the uniform continuity of the maps $(\tau,z)\in L\times\overline{\D} \mapsto \Phi_{\tau}(z)$, $z\in \overline{\D}\mapsto \Phi_w^{-1}(z)$ and $z\in \overline{\D}\mapsto \vp_m(z)$ implies the existence of $\delta(k,m,n)\in (0,+\infty)$ such that (ii) and (iii) hold (with $w$ instead of $w_l(k,m,n)$ and $\delta(k,m,n)$ instead of $\delta_l(k,m,n)$). To obtain the desired sequence it is then enough to cover $L$ by finitely many discs of the form $D(w,\delta(k,m,n))$ and to exhaust $\D$ by a sequence of compact sets.
		
		Then, let a sequence $(w_l(k,m,n),\delta_l(k,m,n))_l$ in $\D\times (0,+\infty)$ be given by the claim. Denoting by $\UU$ the set of all functions $f\in H(\D)$ such that $f\circ \Phi_w \in \UU_A(\D)$ for any $w\in \D$, we also claim that
		\begin{equation}\label{eq-all-auto}
			\UU\supset \bigcap_{k,m,n,l\in \N}\bigcup_{k'\geq k}\left\{f\in H(\D):\, \sup_{z\in F(k',m,n,l)}|f(z)-\vp_m(\Phi_{w_l(k',m,n)}^{-1}(z))|<\frac{1}{m}\right\}.
		\end{equation}
		Indeed, let $f$ be in the right-hand side of the above inclusion and fix $\varepsilon>0$, a compact set $K$ in $\T$ different from $\T$, $\vp \in C(K)$ and $w\in \D$. First let $k,n,m\in \N$ be such that $K\subset K_n$, $|\vp(r_{k'}z)-\vp(z)|<\varepsilon/3$ for any $z\in K_n$ and any $k'\geq k$, $1/m < \varepsilon/3$ and $\sup_{z\in K_n}|\vp_m(r_{k'}z)-\vp(r_{k'}z)|< \varepsilon/3$ for every $k'\geq k$. Note that the choice of $m$ can be done by an application of Mergelyan's theorem. Now, let $k'\geq k$ be given by the fact that $f$ belongs to the right-hand side of $\eqref{eq-all-auto}$. By (i), there exists $l'\in \N$ such that $w\in D(w_{l'}(k',m,n),\delta_{l'}(k',m,n))$. Thus we get for every $z\in K$,
		\begin{equation*}\begin{array}{ccc}\left|f\circ \Phi_w(r_{k'}z) - \vp(r_{k'}z)\right| & \leq & \left|f\circ \Phi_w(r_{k'}z) - \vp_m(r_{k'}z)\right| + \left|\vp(r_{k'}z) - \vp_m(r_{k'}z)\right|\\
				& \leq & \left|f\circ \Phi_w(r_{k'}z) - \vp_m(\Phi_{w_{l'}(k',m,n)}^{-1}(\Phi_{w}(r_{k'}z)))\right| \\ & & \hfill + \left|\vp_m(\Phi_{w_{l'}(k',m,n)}^{-1}(\Phi_{w}(r_{k'}z)))-\vp_m(r_{k'}z)\right| + \frac{\varepsilon}{3}\\
				& \overset{\eqref{eq-all-auto}\text{ and }(iii)}{\leq} & \frac{\varepsilon}{3} + \frac{\varepsilon}{3} + \frac{\varepsilon}{3}.
		\end{array}\end{equation*}
		
		To conclude by the Baire Category Theorem, it is enough to check that the set
		\[
		\bigcup_{k'\geq k}\left\{f\in H(\D):\, \sup_{z\in F(k',m,n,l)}|f(z)-\vp_m(\Phi_{w_l(k',m,n)}^{-1}(z))|<\frac{1}{m}\right\}
		\]
		is open and dense in $H(\D)$ for any $k\in \N$. It is clear that it is open, so let us focus on the density. Let us then fix $k\in \N$, a closed disc $L$ contained in $\D$, $g\in H(\D)$ and $\varepsilon>0$. By uniform continuity of the map $(\tau,z)\in K \times \overline{\D} \mapsto \Phi_{\tau}(z)$ for any compact set $K$ in $\D$, and since automorphisms of $\D$ maps $\T$ into $\T$, there exists $k'\geq k$ such that the set $L\cap F(k',m,n,l)=\emptyset$. Then, by (ii) the set $L\cup F(k',m,n,l)$ is compact and has connected complement. Therefore we can conclude by applying Mergelyan's theorem that gives us a polynomial $P$ that uniformly approximates $g$ in $L$ up to $\varepsilon$, and $\vp_m\circ \Phi_{w_l(k',m,n)}^{-1}$ on $F(k',m,n,l)$ up to $1/m$.
	\end{proof}

	Let us denote by $\UU$ the residual set of all functions in $H(\D)$ such that $f\circ \Phi \in \UU_A(\D,\rho)$ for any automorphism $\Phi$ of $\D$. By definition, $\UU$ is a subset of $\UU_A(\D,\rho)$ and is invariant under composition from the right by any automorphism of the unit disc. Thus, as a corollary of Theorem \ref{left_preserve} and Theorem \ref{thm-all-auto}, we get the following:
	
	\begin{cor}For any $f\in \UU$, any non-constant holomorphic function $g:f(\D)\to \C$, and any automorphism $\Phi$ of $\D$, the function $g\circ f \circ \Phi$ belongs to $\UU$. 
	\end{cor}

	\section*{Acknowledgment}We are thankful to Prof. Alexander Eremenko for bringing our attention to Iversen's theorem. We would also like to thank Dr Dimitris Papathanasiou who suggested us an adaptation of \cite[Theorem 1.6]{BesPapathanasiou2020} in order to prove Theorem \ref{alg-papath}.
	
	\bibliographystyle{amsplain}
	\bibliography{refs}
	
\end{document}